\documentclass[12pt]{article}
\usepackage[utf8]{inputenc}
\usepackage{amsmath}
\usepackage{amsfonts}
\usepackage{amsthm}
\usepackage{amssymb}
\usepackage{mathtools}
\usepackage[T1]{fontenc}
\usepackage{pst-node}
\usepackage{tikz-cd}
\usepackage{caption}
\usepackage{bbm}
\usepackage{comment}
\usepackage{array}   
\newcolumntype{L}{>{$}l<{$}}
\usepackage{multirow}
\usepackage{txfonts}
\usepackage{hyperref}
\usepackage{adjustbox}

\usepackage{tocloft}

\usetikzlibrary{calc,backgrounds}

\pgfdeclarelayer{background}
\pgfdeclarelayer{foreground}
\pgfsetlayers{background,main,foreground}

\usepackage{geometry}
\theoremstyle{plain}
\newtheorem{thm}{Theorem}[section]

\newtheorem{lem}[thm]{Lemma}
\newtheorem{prop}[thm]{Proposition}
\newtheorem{cor}[thm]{Corollary}

\newtheorem{ques}{Question}
\newtheorem{rem}{Remark}

\newtheorem{mythm}{Theorem}

\theoremstyle{definition}
\newtheorem{defn}[thm]{Definition}
\newtheorem{example}[thm]{Example}

\newcommand{\R}{\mathbb{R}}

\newcommand{\Z}{\mathbb{Z}}
\newcommand{\Q}{\mathbb{Q}}
\newcommand{\na}{\mathbb{N}}

\renewcommand{\Im}{\mathrm{im} \, }

\newcommand{\diag}{\mathrm{diag}}

\renewcommand{\cong}{\mathrm{\simeq}}

\newcommand{\g}{\gamma}
\newcommand{\G}{\Gamma}
\renewcommand{\d}{\delta}
\newcommand{\D}{\Delta}

\renewcommand{\l}{\lambda}

\newcommand{\s}{\sigma}
\renewcommand{\S}{\Sigma}
\renewcommand{\t}{\tau}
\renewcommand{\H}{\mathcal{H}}

\newcounter{antoniocomments}

\newcounter{piotrcomments}

\title{Cocycles and positive functionals in higher cohomology}

\author{Antonio López Neumann and Piotr W. Nowak}

\date{}

\begin{document}

\maketitle

\begin{abstract}
We establish and explore the correspondence between positive functionals and cocycles in higher unitary cohomology. We generalize the classical cocycle version of the Gelfand-Naimark-Segal construction to higher degrees and apply it to characterize vanishing of higher unitary cohomology as an extension property for positive functionals. We also prove that under mild conditions the algebraic spectral gap for the one sided Laplacian characterizes cohomological vanishing instead of reducedness of unitary cohomology.


\end{abstract}


\section*{Introduction}

    Property $(T)$ is a central notion in analytic group theory with many applications to other areas of mathematics, including operator algebras, combinatorics, ergodic theory or dynamics. 
   For a group $\G$ generated by a finite symmetric generating set $S$, the Laplacian
       $$ \D = 1 - \frac{1}{|S|} \sum_{s \in S} s,$$
viewed as an element of the real group ring  $\R \G$, plays a central role in studying property $(T)$. 
It is easy to see that the group $\G$ has property $(T)$ if and only if for every unitary representation $\pi$ of $\G$, zero is an isolated point in the spectrum of the Laplacian $\pi(\D)$  \cite[5.4.5]{bekka-harpe-valette}.

    In \cite{ozawa-property-T}, Ozawa showed a remarkable fact, that for groups with property $(T)$ this spectral gap is in fact witnessed on the level of the group ring $\R\G$, via the equation
    $$\D ^2 - \l \D =\sum_{i=1}^n\xi_i^*\xi_i,$$ 
    for some $\l >0$ and some finite collection of $\xi_i\in \R\G$.
    The above characterization enables using computer 
    assistance in proving property $(T)$. This approach allowed later to prove property $(T)$ for $\mathrm{Aut}(F_n)$ for $n \geq 5$ \cite{kaluba-nowak-ozawa, kaluba-kielak-nowak} as well as n=4 \cite{nitsche}.

In recent years there has been increasing interest in generalizing property $(T)$-type phenomena to higher cohomology, see e.g. \cite{bader-nowak}, \cite{DCGLT}, \cite{bader-sauer}, \cite{bader-sauer-survey}.
Property $(T)$ has a couple of equivalent cohomological descriptions, one in terms of the first cohomology $H^1(\G,\pi)$ being reduced (i.e. the image of the codifferential being closed) for every unitary representation
$\pi$ of $\G$, and the other in terms of vanishing of $H^1(\G,\pi)$, again for every unitary representation $\pi$. The equivalence between these two conditions is known as the Delorme-Guichardet theorem (see \cite[2.12.4]{bekka-harpe-valette}). 
This allows for two different routes of generalizing property $(T)$ to higher cohomology, one via reducedness, and another one via vanishing, see \cite[Section 5.1]{bader-nowak1} and \cite{bader-sauer}.
As it turns out, in general these conditions are no longer equivalent for higher cohomology.
Indeed, for a lattice $\G$ in $\mathrm{PGL}_{n+1}(\Q_p)$ the cohomology spaces $H^n(\G,\pi)$ are always Hausdorff but not necessarily zero \cite[Proposition 19]{bader-nowak}.

In \cite{bader-nowak} a characterization of vanishing of unitary cohomology in higher degrees was proved in terms of an algebraic spectral gap of the cohomological Laplacian $\D_n=d_n^*d_n+d_{n-1}d_{n-1}^*$, where $d_n$ is the appropriate group codifferential. 
This characterization in particular allowed again for the use of computer-assisted methods for proving vanishing of cohomology with unitary coefficients.
This approach was implemented in \cite{kaluba-mizerka-nowak}.

For the parallel question of whether unitary cohomology is always reduced, there is an analytic characterization in terms of the spectral gap of the one-sided Laplacian $\D_n^+=d_n^*d_n$ , see \cite{bader-nowak}.
However, its algebraic counterpart, the condition that $\Delta_n^+$ has algebraic spectral gap, remained only a sufficient condition for the cohomology being reduced. 
The question whether reducedness of unitary cohomology can be characterized in terms of positivity in the matrices over the group algebra served as starting point of this work. 
In particular, it was expected that the analytic characterization 
of reducedness from \cite{bader-nowak} has an algebraic counterpart. 
The following theorem gives the unexpected result that this is not the case and the algebraic version of the characterization from \cite{bader-nowak} is sufficiently strong to imply vanishing of cohomology. 

    \begin{mythm}\label{Intro Theorem: Higher Ozawa characterization for higher Vanishing} (Theorem \ref{theorem: Higher Ozawa equation is equivalent to Higher (T)})
    Let $\G$ be a group of type $F$ and $n \geq (\mathrm{cdim}_\mathbb{Q}(\G)+1) /2$. The following conditions are equivalent: 
    \begin{enumerate}
        \item[(1)]  $H^n (\G, \pi) = 0$ for every unitary representation $(\pi, \H)$ of $\G$.
        \item[(2)] There exists $\l >0$ and a finite collection $M_i\in\mathbb{M}_{k_{n-1}} (\R \G)$ such that 
        $$(\D_{n-1}^+)^2 - \l \D_{n-1}^{+} = \sum M_i^*M_i.$$
    \end{enumerate}
    \end{mythm}
The above result gives a complete higher dimensional generalization of Ozawa's characterization of property $(T)$. See Theorem \ref{theorem: Higher Ozawa equation is equivalent to Higher (T)} for a version with more general hypotheses.

In order to prove the above, new techniques were necessary and this had led us to exploring the connection between positive functionals, their extensions, and cocycles in group cohomology. As a result, we establish a higher-dimensional version of the celebrated Gelfand-Naimark-Segal construction. For a group $\G$, the classical cocycle GNS construction provides 
a correspondence between positive definite functionals on the augmentation ideal on one side, and 1-cocycles of unitary representations on the other. The correspondence is expressed via the formula 
$$\phi(1- \g)= ||z(\g)||^2 ,$$ 
 where $\g \in \G$, $\pi$ is a unitary representation, $z$ an associated 1-cocycle and $\phi$ is a positive functional.

In order to generalize the above cocycle version of the GNS construction we consider a certain chain complex on algebras of matrices $\mathbb{M}_n(\R\G)$ with coefficients in $\R\G$,  induced by the group cochain complex with coefficients in $\R\G$.
This matricial chain complex is not acyclic in general and it is a natural question when does its homology vanish. 
For us, matricial cycles and boundaries serve as a natural environment for generalizing the GNS construction. 
Our higher GNS construction, formulated in Theorems \ref{higher GNS} and \ref{higher GNS - image version}, gives a correspondence between $n$-cocycles $z$ in the cohomology $H^n(\G,\pi)$, where $\pi$ is a unitary representation, and functionals defined by the matrix coefficient $\langle \pi(a)z,z\rangle$:

 \begin{center}
    \begin{tikzcd}[scale=0.7em]
    
    \begin{Bmatrix}
         \phi \text{ positive functional on}   \\
      \text{ matricial boundaries}
    \end{Bmatrix}
    
    \arrow[r, leftrightarrow] 
    
    	 & \begin{Bmatrix}
         z \in Z^n (\G, \pi),   \\
        \pi \text{ unitary $\G$-representation}
    \end{Bmatrix} 
  \end{tikzcd}
\end{center}

An important technical element, formulated as Theorem \ref{theorem: ker D_i=SC_i}, is a description of matricial cycles as sums of images of one-sided differentials that we call \emph{split-conjugation spaces}  (see Definition \ref{Definition: split conjugation}). For this we use a homological telescoping argument together with acyclicity 
of the underlying group chain complex. 

As one of the consequences we prove a new characterization of vanishing of higher unitary cohomology via an extension property for positive functionals defined on the matricial coboundaries.

\begin{mythm} (Theorem \ref{Vanishing of cohomology vs extension of functionals}) Let $\Gamma$ be a group of type $F$ and $n \geq (\mathrm{cdim}_\mathbb{Q}(\G)+1) /2$. The following conditions are equivalent:
    \begin{enumerate}
        \item[(1)]  $H^n (\G, \pi) = 0$ for every unitary representation $(\pi, \H)$ of $\G$.
    \item[(2)]  Every positive hermitian functional $\phi:\operatorname{im} D_{n-1} \to \R$ can be extended to a positive hermitian functional $\overline{\phi}:\mathbb{M}_{k_{n-1}}(\R\G)\to\R$.
    \end{enumerate}
\end{mythm}
Above, $D_{n-1}$ is the appropriate matricial differential operator and its image is the space of matricial boundaries. 

We also apply our methods to discuss generalizations of  Shalom's property $H_T$ 
to higher degrees. Namely, we obtain characterizations of a higher version of property $H_T$ in the spirit of \cite[Corollary 12]{ozawa-substitute}. See Theorems \ref{Characterization of higher H_T} and \ref{Higher analogue of Ozawa's characterization}.

 \subsection*{Acknowledgements.} 

Both authors were supported by the National Science Center Grant Maestro-13 UMO-2021/42/A/ST1/00306.
The first author was also supported by Fondation Sciences Mathématiques de Paris.
The second author would like to thank Piotr Mizerka for many discussions on topics related to the subject of this paper. 

We thank Uri Bader and Saar Bader for their interest in this work and helpful comments.

\tableofcontents

\section{Matrices over group algebras}

We first collect some definitions and notations on $*$-algebras. We refer to \cite{bader-nowak}, \cite{cimprivc2009}, \cite{netzer-thom} or \cite{schmudgen-book} for details.

A \textit{$*$-algebra} $A$ is an algebra over $\R$ equipped with an involution $x \mapsto x^*$ that is $\R$-linear and such that $(ab)^* = b^*a^*$ for $a, b \in A$. For a subspace $V\subseteq A$ we define the cone of \textit{sums of hermitian squares} by
\begin{equation*}
    \S^2V : = \left\lbrace\sum_{i=1}^n a_i^* a_i \colon a_i \in V,  n \in \na \right\rbrace.
\end{equation*}

Given a $*$-algebra $A$ and a subspace $V\subseteq A$, we denote by $V^h = \{ a \in V : a^* = a \}$ the subspace of \textit{self-adjoint} elements of $V$. We have $\S^2 V \subset V^h$. An \textit{order unit} $\d \in \S^2 A$ is an element such that for every $a \in A^h$ there exists $R>0$ such that $a + R \d \in \S^2 A $.

We say that a $*$-algebra is \textit{Archimedean} if it is unital (with unit denoted by $1$), $-1 \notin \S^2 A$ and $1$ is an order unit of $A$.

Given two $*$-algebras (resp. unital $*$-algebras) $A$ and $B$, a map $f: A \to B$ is a \emph{$*$-homomorphism} (resp. a \emph{unital $*$-homomorphism}) if it is an algebra homomorphism and $f(a^*) = f(a)^*$ for $a \in A$ (resp. satisfying furthermore $f(1_A) = 1_B$).

Given a $*$-algebra $A$, we say that a linear functional $\varphi: A \to \R$ is \textit{hermitian} if $\varphi(a ) = \varphi(a^*)$ for every $a \in A$. Such a functional is characterized by its restriction to $A^h$. We say that a linear functional $\varphi: A \to \R$ is \textit{positive} if $$\varphi (\S^2 A) \geq 0.$$ 

Our $*$-algebras will be endowed with the so-called \textit{algebraic topology} as in \cite{ozawa-property-T}, defined as follows.
Let $V$ be a real vector space and $C \subset V$ a convex subset. We say that $c \in C$ is an \textit{algebraic interior point} if for every $v \in V$, there exists $t \in (0,1]$ such that $(1-t)c + tv \in C$. We denote by $\mathrm{int }(C)$ the set of all interior points of $C$. We define a locally convex topology on $V$ by defining open sets to be all convex subsets $C$ such that $C = \mathrm{int}(C)$. In this topology all linear functionals are continuous and all vector subspaces are closed.

If we endow a $*$-algebra $A$ (or $A^h$) with this topology and $\d$ is an order unit such that $- \d \notin \S^2 A$, we have 
\begin{equation*}
    \overline{\S^2 A} = \{ a \in A^h\colon  \forall \varepsilon>0,   a + \varepsilon \d \in \S^2 A \}
\end{equation*}
and if $W$ is a real vector subspace of $A^h$, then
\begin{equation*}
    \overline{\S^2 A + W} = \{ a \in A^h,   \forall \varepsilon>0,   a + \varepsilon \d \in \S^2 A + W \}.
\end{equation*}


We now turn to group algebras.
Recall that given a countable discrete group $\G$, the group algebra $\R \G$ is the vector space over $\R$ of finite linear combinations $\sum_{\gamma\in \G} a_\gamma \gamma$ of elements of $\G$, with the product given by the linear extension of the group product. It is a unital $*$-algebra when endowed with the involution $\sum a_\g \g \in \R \G \mapsto \sum a_\g \g^{-1}$ and unit $e_\G$. The subspace 
$$ I (\G) := \left\{ \sum_{\g \in \G} a_\g \g \in \R \G \colon \sum_{\g \in \G} a_\g  = 0 \right\} $$
is an ideal, called the \emph{augmentation ideal} of $\R \G$.
One can show that $\R \G$ is an Archimedean $*$-algebra \cite[Example 3]{cimprivc2009} and hence for $k \geq 1$ the algebra $\mathbb{M}_k(\R \G)$ 
of matrices with coefficients in the group algebra is also an Archimedean $*$-algebra \cite[Corollary 13]{bader-nowak}, with $*$-involution given by taking the transpose and applying the $*$-involution of $\R \G$ on each entry.

A \textit{$*$-representation} $(\pi, \H)$ of $\mathbb{M}_k(\R \G)$ consists of a Hilbert space $\H$ and a unital $*$-homomorphism $\pi : \mathbb{M}_k(\R \G) \to \mathcal{L}(\H)$ with values in the $*$-algebra of linear operators on $\H$. We say that the $*$-representation $\pi : \mathbb{M}_k(\R \G) \to \mathcal{L}(\H)$ is \textit{bounded} if the image $\pi(\mathbb{M}_k(\R \G))$ lies in the subspace of bounded operators $\mathrm{B}(\H)$. 

The following Proposition says that $*$-representations of $\mathbb{M}_k(\R \G)$ are always induced by unitary representations of the underlying group $\G$.

\begin{prop} \cite[Proposition 12]{bader-nowak} \label{*-reps of matrices decompose as a unirep coordinate-wise}
    Let $k \geq 1$ and $\rho: \mathbb{M}_k(\R \G) \to \mathrm{B}(\mathcal{K})$ be a bounded $*$-representation of $\mathbb{M}_k(\R \G)$. Then there exists a unitary representation $\pi: \G \to \mathrm{U}(\H)$ of $\G$ and a Hilbert space isometric isomorphism $\theta: \mathcal{K} \to \H^k$ such that for $a = (a_{ij}) \in \mathbb{M}_k(\R \G)$ we have
    \begin{equation*}
        \rho(a)  =\theta^{-1}   \circ \pi(a_{ij}) \circ \theta.
    \end{equation*}
\end{prop}

As a corollary, the classical Gelfand-Naimark-Segal construction for positive functionals on matrices over the group algebra can be understood via matrix coefficients for unitary representations of the underlying group.

\begin{prop} [GNS construction] \label{GNS for matrices of group algebras}
      Let $\G$ be a discrete group and $k \geq 1$.
 Let $\psi : \mathbb{M}_k(\R \G) \to \R$ be a positive hermitian linear functional. Then there exist a unitary representation $\pi: \G \to \mathrm{B}(\H)$ of $\G$ and $c \in \H^k$ such that for $a \in \mathbb{M}_k(\R \G)$ we have
\begin{equation*}
    \psi(a) = \langle \pi(a) c, c \rangle_{\H^k}.
\end{equation*}
The vector $c$ is cyclic for the corresponding $\mathbb{M}_k(\R \G)$-representation, that is, the space $\pi (\mathbb{M}_k(\R \G)) c$ is dense in $\H^k$.

(Uniqueness) Let $(\pi ', \H')$ be another unitary representation of $\G$ and $c' \in (\H')^k$ be a vector such that $\psi(a) =  \langle \pi'(a) c', c' \rangle_{(\H')^k}$ for every $a \in \mathbb{M}_k(\R \G)$ and such that the space $\pi' (\mathbb{M}_k(\R \G)) c'$ is dense in $(\H')^k$. Then there exists a unitary isomorphism $U : \H \to \H'$ intertwining $\pi$ and $\pi'$ as unitary $\G$-representations and mapping entries of $c$ to entries of $c'$.

\end{prop}

\begin{proof}
    The classical GNS construction \cite[Theorem 4.38]{schmudgen-book} applied to the unital $*$-algebra $\mathbb{M}_k(\R \G)$ says that there exist a bounded $*$-representation $\rho: \mathbb{M}_k(\R \G) \to \mathrm{B}(\mathcal{K})$ of $\mathbb{M}_k(\R \G)$ and a cyclic vector $\xi \in \mathcal{K}$ such that $$ \psi(a) = \langle \rho(a) \xi, \xi \rangle. $$
    Applying Proposition \ref{*-reps of matrices decompose as a unirep coordinate-wise} to $\rho$ yields a unitary representation $\pi: \G \to \mathrm{U}(\H)$ of $\G$ and a Hilbert space isometric isomorphism $\theta: \mathcal{K} \to \H^k$ such that for $a = (a_{ij}) \in \mathbb{M}_k(\R \G)$ we have $ \rho(a)  =\theta^{-1}   \circ( \pi(a_{ij})) \circ \theta$. By setting $c: = \theta(\xi)$ we obtain
    $$\psi(a) = \langle \rho(a) \xi, \xi \rangle =  \langle \theta^{-1}   \circ \pi(a) \circ \theta (\xi), \xi \rangle = \langle \pi(a) c, c \rangle. $$
    Since $\theta$ is an isometric isomorphism and $\xi \in \mathcal{K}$ is a cyclic vector, we see that $c$ is a cyclic vector in $\H^k$ for the action of $\pi(\mathbb{M}_k(\R \G))$.

   To prove uniqueness, let $(\pi ', \H')$ be a unitary representation of $\G$ and $c' \in (\H')^{k}$ a vector such that $\psi(a) =  \langle \pi'(a) c', c' \rangle_{(\H')^k}$ and such that $\pi'(\mathbb{M}_k(\R \G)) c'$ is dense in $(\H')^k$. Write $c = (c_1, \ldots, c_k)^T$ and $c' = (c_1', \ldots, c_k')^T$.
    The orbit $\pi(\mathbb{M}_k(\R \G)) c$ can be described using coordinates as $$\pi(\mathbb{M}_k(\R \G)) c = \pi(\mathbb{M}_k(\R \G)) \begin{pmatrix}
         c_1 \\
        \vdots \\
        c_k
    \end{pmatrix} = \begin{pmatrix}
        \sum_{j=1}^k \pi(\R \G) c_j \\
        \vdots \\
        \sum_{j=1}^k \pi(\R \G) c_j
    \end{pmatrix} = \left( \sum_{j=1}^k \pi(\R \G) c_j \right)^k. $$
    Hence, if the space $\sum_{j=1}^k \pi(\R \G) c_j$ was not dense in $\H$, then $\pi(\mathbb{M}_k(\R \G)) c$ would not be dense in $\H^k$.
    Our density assumptions imply that the subspaces
$$ \H_0 := \sum_{j=1}^k \pi(\R \G) c_j, \quad \H'_0 := \sum_{j=1}^k \pi'(\R \G) c_j', $$
are dense in $\H$ and $\H'$ respectively. Consider the maps 
\begin{align*}
    \Phi : (\R \G )^k &\to \H_0 & \Phi' : (\R \G )^k &\to \H_0'   \\
    (\xi_1, \ldots, \xi_k) & \mapsto \sum_{j = 1}^k \pi(\xi_j) c_j  & (\xi_1, \ldots, \xi_k) & \mapsto \sum_{j = 1}^k \pi'(\xi_j) c_j'
\end{align*}
For any $1 \leq i \leq k$, we consider the matrix $$ a_i = \begin{pmatrix}
    0 & \ldots & 0 \\
    \vdots & & \vdots \\
    \xi_1  & \ldots & \xi_k \\
       \vdots & & \vdots \\
    0 & \ldots & 0
\end{pmatrix}$$
where $\xi_j \in \R \G$ for $1 \leq j \leq k$. Since we have $ \langle \pi(a) c, c \rangle_{\H^k} =  \langle \pi'(a) c', c' \rangle_{(\H')^k}$ for every $a \in \mathbb{M}_k(\R \G)$, replacing $a$ by $a_i$ yields
$$ \left\langle \sum_{j = 1}^k \pi(\xi_j) c_j , c_i \right\rangle  =  \left\langle \sum_{j = 1}^k \pi'(\xi_j) c_j' , c_i' \right\rangle. $$ 
From this, using linearity and adjoints, we can show that for $\eta_i \in \R \G$, where $1 \leq i \leq k$, we also have
$$ \left\langle \sum_{j = 1}^k \pi(\xi_j) c_j , \sum_{i = 1}^k \pi(\eta_i) c_i \right\rangle  =  \left\langle \sum_{j = 1}^k \pi'(\xi_j) c_j' , \sum_{i = 1}^k \pi'(\eta_i) c_i' \right\rangle, $$ 
which means that 
$$ \left\langle \Phi (\xi_1, \ldots, \xi_k) ,\Phi (\eta_1, \ldots, \eta_k) \right\rangle  =  \left\langle \Phi' (\xi_1, \ldots, \xi_k) ,\Phi' (\eta_1, \ldots, \eta_k) \right\rangle. $$ 
This implies that $\ker \Phi = \ker \Phi'$, so the map $\Phi' : (\R \G)^k \to \H_0'$ factors through $\Phi:  (\R \G)^k \to \H_0$ via a linear map $U : \H_0 \to \H_0'$ such that $ \Phi' = U \circ \Phi $. Since $\Phi$ is surjective onto $\H_0$, the previous equality implies that $U :\H_0 \to \H_0'$ is unitary, and onto $\H_0'$ as $\Phi'$ is surjective. By density of $\H_0$ and $\H_0'$ in $\H$ and $\H'$ respectively, the map $U$ can be extended to a surjective unitary linear operator $ U : \H \to \H'$.

By choosing vectors $(0, \ldots, 1, \ldots, 0) \in (\R \G)^k$, the relation $\Phi' = U \circ \Phi$ implies that $U c_i = c_i'$ for every $1 \leq i \leq k$. The map $U$ intertwines $\pi$ and $\pi'$, since for every $\xi_j \in \R \G$ we have
$$ U \left( \pi (\g) \sum_{j = 1}^k \pi(\xi_j) c_j \right) = U \left( \sum_{j = 1}^k \pi(\g \xi_j) c_j \right) =  \sum_{j = 1}^k \pi'(\g \xi_j) c_j'= \pi(\g)  \sum_{j = 1}^k \pi'( \xi_j) c_j', $$
so $U$ intertwines $\pi$ and $\pi'$ on $\H_0$ and hence on $\H$, by density.
\end{proof}

\begin{rem}\normalfont
    Note that in the above setting, the $k \times k$
   diagonal matrices $\mathcal{U}^{(k)}=\diag(U,U,\dots, U)$ intertwine
   the representations $\pi$ and $\pi'$ when viewed as a representation of $\mathbb{M}_k(\R\G)$, and also 
   intertwine non-square matrices such as the codifferentials $d_{n}$ by setting 
   $$\mathcal{U}^{(k_{n+1})} \pi(d_n)= \pi'(d_n) \mathcal{U}^{(k_{n})}.$$
\end{rem}

We record some facts on squares of non-squares matrices.

\begin{lem}\label{Squares of non-square matrices}
    Let $\G$ be a discrete group and $k \geq 1$. 
    \begin{itemize}
        \item[(1)] \cite[Lemma 14]{bader-nowak} Let $x \in \mathbb{M}_{k' \times k}(\R \G)$. Then $x^* x \in \S^2 \mathbb{M}_k(\R \G)$.
        \item[(2)] Let $x_i \in \mathbb{M}_{k' \times k}(\R \G)$ for $i \in I$, where $I$ is finite. If $\sum_{i \in I} x_i^*x_i = 0\in \S^2 \mathbb{M}_k(\R \G)$ then $x_i = 0$ for every $i \in I$.
    \end{itemize}
    
\end{lem}

\begin{proof}
We will only prove $(2)$. If $x_i \in \mathbb{M}_{k' \times k}(\R \G)$ satisfy $\sum_{i \in I} x_i^*x_i = 0$, we have:
    \begin{equation*}
       \sum_{i \in I} ||x_i||_{\ell^2(\mathbb{M}_{k' \times k}(\R \G))}^2 = \left\langle \sum_{i \in I} x_i^*x_i, I_n \right\rangle = 0.
    \end{equation*}
    Hence $x_i= 0$ for every $i \in I$.

\end{proof}

\section{Group cohomology and finiteness properties}

We refer to standard references for elementary definitions of cohomological nature such as group cohomology or geometric dimension \cite{brown-cohomology, borel-wallach}.

We say that a group $\G$ is of \textit{type $F_{m}$} for some $m\geq 0$ if $\G$ admits a free action by simplicial automorphisms on some locally finite contractible simplicial complex $X$, such that for every $n \leq m$ the induced action of $\G$ on the set of $n$-simplices $X^{(n)}$ of $X$ has a finite number of $\G$-orbits $k_n$.

Let $\G$ be a group of type $F_{m+1}$ for some $m \geq 0$ and $X$ be a simplicial complex as above.
The space of $\R \G$-valued $\G$-equivariant $n$-cochains is denoted by $$C^n(X, \R \G)^\G := \{  c: X^{(n)} \to \R \G, c(\g \s) = \g c(\s) \in \R \G, \g \in \G, \s \in X^{(n)} \}.$$
Notice that by choosing representatives for the $\G$-orbits in $X^{(n)}$, we can identify $C^n(X, \R \G)^\G$ and $(\R \G)^{k_n}$. By tensoring with the identity, the standard simplicial coboundary operator on $X$ extends naturally to an $\R \G$-equivariant map 
\begin{equation*}
    d_n: C^n(X, \R \G)^\G \to C^{n+1}(X, \R \G)^\G.
\end{equation*}
Hence we may view the operator $d_n$ as a matrix $d_n \in \mathbb{M}_{k_{n+1} \times k_{n}}(\R \G)$ for $n \leq m$. 
The adjoint matrix $d_n^*  \in  \mathbb{M}_{k_{n} \times k_{n+1}}(\R \G)$ of $d_n$ is just the transpose of the matrix $d_n$ after applying the $*$-involution of $\R \G$ on each entry. The adjoint $d_n^*$ can be obtained by the same method from the simplicial boundary operator on $X$.

The above setting is assumed for the following lemma.

\begin{lem}\label{left mult by d* is acyclic/ right mult by d is acyclic} 
    For any $j\geq 1$, the chain complex $d^*_n \square$, given in degrees $n \leq m$ by left multiplication by the matrices $d^*_n$,
    \begin{equation*}
 0 \xleftarrow{} \mathbb{M}_{1\times j}(\R \G) \xleftarrow{d^*_0 \square} \mathbb{M}_{k_1 \times j}(\R \G)  \xleftarrow{d^*_1 \square } \mathbb{M}_{k_2 \times j}(\R \G) \xleftarrow{} \ldots
    \end{equation*}
    is acyclic in degrees $1, \ldots, m$. Similarly, for any $j\geq 1$ the chain complex $\square d_n$ of right multiplication by the matrices $d_n$ in degrees $n \leq m$,
        \begin{equation*}
 0 \xleftarrow{}\mathbb{M}_{j \times 1}(\R \G) \xleftarrow{ \square d_0} \mathbb{M}_{ j \times k_1}(\R \G)  \xleftarrow{\square d_1} \mathbb{M}_{j \times k_2}(\R \G) \xleftarrow{} \ldots
    \end{equation*}
    is acyclic in degrees $1, \ldots, m$.
\end{lem}

\begin{proof}
We first prove the result for $j = 1$. Notice that $ \mathbb{M}_{k_n \times 1} (\R \G) \simeq C_n(X, \R)$ for $n \leq m+1$, the space of simplicial cochains, and that the operator $d_n^*$ corresponds to the usual simplicial boundary operator $\partial_n$. This identification is given explicitly by numbering $n$-simplices in a fundamental domain of $X$ for the $\G$-action by $\t_i$ for $1\leq i \leq k_{n+1}$ and associating the chain $c = \sum_{\g \in \G, 1\leq i \leq k_{n}} v_i(\g) \g. \t_i$ to the vector $v \in \mathbb{M}_{k_{n} \times 1} (\R \G)$ . This means that the chain complex given by
    \begin{equation*}
          0 \xleftarrow{}\R \G \xleftarrow{d^*_0} \mathbb{M}_{k_1 \times 1}(\R \G)  \xleftarrow{d^*_1} \mathbb{M}_{k_2 \times 1}(\R \G) \xleftarrow{} \ldots   
    \end{equation*}
has the same homology as the complex given by $C_*(X, \R)$, and hence is acyclic in degrees $1, \ldots, m$. This proves the first statement for $j = 1$.

Now let $j \geq 2, 1 \leq n \leq m+1$ and $a \in  \mathbb{M}_{k_n \times j}(\R \G)$ such that $d_{n-1}^* a = 0$. We write $a = (a_i)_{1\leq i \leq j}$ as a horizontal vector, where each $a_i \in \mathbb{M}_{k_n \times 1} (\R \G)$ corresponds to the $i$-th column of $a$. Note that when $d_{n-1}^*$ acts on the left of $a$, it acts separately on each column $a_i$. Hence the condition $d_{n-1}^*a^*=0$ is equivalent to $d_{n-1}^* a_i = 0$ for every $1 \leq i \leq j$. There exist $b_i \in \mathbb{M}_{k_{n+1} \times 1} (\R \G)$ such that $a_i = d_n^* b_i$ for each $1 \leq i \leq j$. Hence the matrix $b = (b_i)_{1 \leq i \leq j} \in  \mathbb{M}_{k_{n+1} \times j}(\R \G)$ satisfies $a = d_n^* b$ and the assertion is proved.

The statement about the complex $\square d_n$ given by right multiplication by $d_n$ follows, as it suffices to apply the involution $*$ and use the acyclicity for $d_n^* \square$.
\end{proof}

Let $(\pi, \H)$ be a unitary $\G$-representation on a Hilbert space $\H$. We define $C^n(X, \H)^\G$ to be the set of $\G$-equivariant maps $X^{(n)} \to \H$. As before, choosing representatives for the $\G$-orbits in $X^{(n)}$ allows us to identify $C^n(X, \H)^\G$ and $\H^{k_n}$. 
Under this identification, we can see the coboundary operator $C^n(X, \H)^\G \to C^{n+1}(X, \H)^\G$ as the matrix $\pi(d_n) \in \mathbb{M}_{k_{n+1} \times k_{n}}(B (\H))$, where one applies the representation $\pi$ entry-wise to the the matrix $d_n \in \mathbb{M}_{k_{n+1} \times k_{n}}(\R \G)$ defined earlier.
The matrices $\pi(d_n)$ compute group cohomology, that is, for every $n \geq 0$ we have:
    \begin{align*}
        H^n(\G, \pi) & = \ker \pi(d_n) / \Im \pi(d_{n-1}), \\
        \overline{H}^n(\G, \pi) & = \ker \pi(d_n) / \overline{\Im \pi(d_{n-1})}
    \end{align*}
(see for instance \cite[Proposition 1]{bader-nowak}).
An important tool for computing unitary cohomology is simplicial Hodge theory, that is, the fact that we can compute it via harmonic cochains. Let
\begin{align*}
    &\D_n^+ :=  d_n^*d_n \in \mathbb{M}_{k_n}(\R \G) &\text{ for } n \geq 0,  \\
    &\D_n^- := d_{n-1} d_{n-1}^* \in \mathbb{M}_{k_n}(\R \G) &\text{ for } n \geq 1, & \;\D_0^- = 0 \in \R \G, \\
     &\D_n = \D_n^+ +\D_n^- \in \mathbb{M}_{k_n}(\R \G) &\text{ for } n \geq 0 . 
\end{align*}

Let $(\pi, \H)$ be a unitary representation of $\G$ and $n \geq 0$. We have continuous isomorphisms:
\begin{equation}\label{reduced cohom = harmonic cocycles}
    \overline{H}^n(\G, \pi) \cong  \ker \pi(\D_n),
\end{equation}
see e.g. \cite[Proposition 16 (1)]{bader-nowak}.

\section{Split conjugation spaces}

Let $\G$ be a group of type $F_m+1$ for some $m \in \na$. Consider the cochain complex from the previous section.

\begin{center}
\begin{tikzcd}
    \cdots \arrow{r} & \R\G^{k_{n-1}} \arrow{r}{d_{n-1}} &  \R\G^{k_{n}} \arrow{r}{d_{n}} & \R\G^{k_{n+1}}  \arrow{r}{d_{n+1}} & \cdots .
\end{tikzcd}
\end{center}
We will consider the induced complex of square matrices with coefficients in $\R \G$
\begin{center}
\begin{tikzcd}
    \cdots
    &\arrow{l}\mathbb{M}_{k_{n-1}} (\R\G)  
    & \arrow{l}{\ \ \ D_{n-1}} \mathbb{M}_{k_{n}} (\R\G) 
    & \arrow{l}{\ \ \ D_{n}}\mathbb{M}_{k_{n+1}} (\R\G)
    &\arrow{l} \cdots, 
\end{tikzcd}
\end{center}
where the maps $D_n$ are defined as 
\begin{equation*}
D_n(a)= d_n^* a d_n,
\end{equation*}
for  $a \in \mathbb{M}_{k_{n+1}}(\R\G)$.

We first prove a property of positive elements in the images of the maps $D_i$, namely that the matricial chain complex is acyclic after restriction to positive elements.

\begin{prop}\label{SOS in Im D_k are images of SOS}
Let $\G$ be of type $F_{n+1}$ for some $n\geq 1$. We have
$$D_n\left( \Sigma^2 \mathbb{M}_{k_{n+1}}(\R\G)\right)= \operatorname{im} D_n \cap \Sigma^2\mathbb{M}_{k_n}(\R\G) = \operatorname{ker} D_{n-1} \cap \Sigma^2\mathbb{M}_{k_n}(\R\G).$$
\end{prop}

\begin{proof} The inclusion $D_n\left( \Sigma^2 \mathbb{M}_{k_{n+1}}(\R\G)\right) \subseteq \operatorname{im} D_n \cap \Sigma^2\mathbb{M}_{k_n}(\R\G)$ follows from Lemma \ref{Squares of non-square matrices} $(1)$. The inclusion $\operatorname{im} D_n \cap \Sigma^2\mathbb{M}_{k_n}(\R\G) \subseteq \operatorname{ker} D_{n-1} \cap \Sigma^2\mathbb{M}_{k_n}(\R\G)$ follows from $D_n D_{n-1} = 0$.
We will verify that for a finite collection of $a_i \in \mathbb{M}_{k_n}(\R \G)$, the condition $D_{n-1}\left(\sum_{i}a_i^*a_i\right) = 0$ implies 
$$\sum_i a_i^*a_i = D_{n}(b),$$ 
for some $b \in \S^2 \mathbb{M}_{k_{n+1}}(\R \G)$. 
By Lemma \ref{Squares of non-square matrices} $(2)$,
the equality 
$$0 =  D_{n-1}\left(\sum_i a_i^*a_i\right) = \sum_i d_{n-1}^* a_i^* a_i d_{n-1}= \sum_i (a_i d_{n-1})^* a_i d_{n-1}$$ 
is equivalent to $ a_i d_{n-1} = 0$ for all $i$. By Lemma \ref{left mult by d* is acyclic/ right mult by d is acyclic}, there exist $b_i \in \mathbb{M}_{k_{n} \times k_{n+1}}(\R \G)$ such that $a_i = b_i d_n$ for all $i$. Hence 
 $$ \sum_i a_i^*a_i = \sum_i d_n^* b_i^*b_i d_n = d_n^* \left( \sum_i  b_i^*b_i \right)  d_n =  D_n \left( \sum_i  b_i^*b_i \right),$$
 as required.
\end{proof}

In our further investigation it will be important to describe the kernels of the maps $D_n$ in
a certain way. The spaces we introduce here can be viewed as higher-dimensional analogues of the augmentation ideal, as 
$\ker D_{-1}$ is in fact precisely the augmentation ideal $I(\G)$. 

\begin{defn}\label{Definition: split conjugation}
The \emph{$i$-th split conjugation space} $\mathcal{SC}_i$ is defined by
$$\mathcal{SC}_i: =\left\lbrace a \in \mathbb{M}_{k_i}(\R\G) : a=d_i^*x+yd_i, x \in \mathbb{M}_{k_{i+1} \times k_{i}}(\R \G), y \in \mathbb{M}_{k_{i} \times k_{i+1}}(\R \G)     \right\rbrace. $$

\end{defn}

Notice that 
$$ \Im D_i \subseteq \mathcal{SC}_i \subseteq \ker D_{i-1}. $$
Indeed, for $a = d_i^* x d_i$ where $x \in \mathbb{M}_{k_{i+1}}(\R \G))$, we have $a = \frac{1}{2}d_i^* (x d_i) + \frac{1}{2} (d_i^* x) d_i$. For the inclusion $ \mathcal{SC}_i \subseteq \ker D_{i-1}$, let $a = d_{i}^* x + y d_{i}$ for some $ x \in \mathbb{M}_{k_{i+1} \times k_{i}}(\R \G)$ and $ y \in \mathbb{M}_{k_{i} \times k_{i+1}}(\R \G) $. We have
    \begin{equation*}
        d_{i-1}^* a d_{i-1} = d_{i-1}^* (d_{i}^* x + y d_{i}) d_{i-1} = (d_{i-1}^* d_{i}^*) x  d_{i-1} + d_{i-1}^*y (d_{i} d_{i-1}) = 0.
    \end{equation*}

The following Theorem gives a criterion to ensure that the last inclusion is an equality.

\begin{thm}\label{theorem: ker D_i=SC_i}

Let $\Gamma$ be a group of type $F$ and geometric dimension $\leq n$. Then $$\ker D_{i-1} = \mathcal{SC}_{i},$$ for $i\ge (n+1)/2$.
\end{thm}

\begin{proof}
Suppose $i \geq (n +  1)/2$ and let $a \in \ker D_{i-1}$; that is, $d_{i-1}^* a d_{i-1} = 0$. By Lemma \ref{left mult by d* is acyclic/ right mult by d is acyclic}, there exists $a_1 \in \mathbb{M}_{k_{i+1}\times k_{i-1}}(\R \G)$ such that 
    $$ad_{i-1} = d^*_{i}a_1.$$ 
    Hence $$d^*_{i}a_1 d_{i-2} = a d_{i-1} d_{i-2} = 0.$$ We can iterate this argument and construct a sequence of matrices $(a_j)$, where $a_0  = a$ and $a_j \in \mathbb{M}_{k_{i+j} \times k_{i- j}} (\R\G)$ is such that 
    $$a_j d_{i-j-1} = d_{i+j}^* a_{j+1},$$ 
    provided $i -j -1 \geq 0$. 

\begin{center}
\adjustbox{scale=0.9}{
\begin{tikzcd}
    \dots
    &\arrow{l}\R\G^{k_{i-2}}   \arrow[bend left=35]{rrrr}{a_2}
    &\arrow{l} \R\G^{k_{i-1}} \arrow[bend left=30]{rr}{a_1} 
    & \arrow{l}{\ \ \ d_{i-1}^*} \R\G^{k_{i}} 
    & \arrow{l}{\ \ \ d_i^*}  \R\G^{k_{i+1}}
    & \arrow{l}{\ \ \ d_{i+1}^*} \R\G^{k_{i+2}} 
    &\arrow{l} \cdots, 
\end{tikzcd}
}
\end{center}
\vspace{12pt}

    Since $i \geq (n + 1)/2$, we can apply this procedure inductively up to $j = n- i$, where we have 
    $$a_{n-i} d_{2i - n - 1} = d^*_{n} a_{n-i+ 1}.$$
    But the geometric dimension of $\G$ is at most $n$, so we have $d^*_{n} = 0$ and hence 
    $$a_{n-i} d_{2i - n-1} = 0.$$
    Now, by Lemma \ref{left mult by d* is acyclic/ right mult by d is acyclic}, there exists $b_{n-i} \in \mathbb{M}_{k_n \times k_{2i - n + 1}}(\R\G)$ such that $a_{n-i} = b_{n-i} d_{2i - n}$.
    Since we also have 
    $$a_{n - i - 1}d_{2i - n} = d_{n-1}^* a_{n-i} = d_{n-1}^*  b_{n-i} d_{2i - n},$$
    we obtain
    $$ (a_{n-i -1} - d_{n-1}^*b_{n-i}) d_{2i- n} = 0. $$
    By Lemma \ref{left mult by d* is acyclic/ right mult by d is acyclic} there exists $b_{n-i-1}\in \mathbb{M}_{k_{n-1} \times k_{2i - n+ 2}}(\R\G)$ such that
    $$ a_{n-i -1} = d_{n-1}^*b_{n-i} + b_{n-i -1} d_{2i- n +1}. $$
    Iterating this argument we obtain a sequence $b_j\in \mathbb{M}_{k_{i+j} \times k_{i-j+1}}(\R\G)$ (descending in $j$) for $ 0 \leq j \leq n - i - 1$ such that
    $$ a_j = d_{i+j}^*b_{j + 1} + b_j d_{i- j}.$$
    Reaching $j = 0$ we obtain  
    $$ a = a_0 = d_{i} b_1 + b_0 d_{i}.$$
    Hence $a \in \mathcal{SC}_{i}$.
\end{proof}

\begin{rem}\normalfont
    This proof and its consequences remain valid after replacing property $F$ and geometric dimension by their corresponding algebraic counterparts, namely, property $FP$ and cohomological dimension.
\end{rem}

\section{The higher dimensional GNS construction}

The Gelfand-Naimark-Segal construction is a fundamental tool in representation theory
and cohomology with unitary coefficients. One classical application of this construction, is that it gives a correspondence between 1-cocycles of unitary representations and conditionally negative functions (that we view as positive functionals on the augmentation ideal). We first recall the statement.

\begin{prop} [Classical cocycle GNS] \label{Classical conditionally negative GNS} 
Let $\G$ be a discrete group.
    \begin{enumerate}
        \item[(1)] For every hermitian positive functional $\phi : I(\G) \to \R$, there exists a unitary representation $(\pi, \H)$ of $\G$ and a 1-cocycle $z \in Z^1 (\G, \pi)$ such that $$\phi (1 - \g) = ||z (\g)||^2.$$
        \item[(2)] For every unitary representation $(\pi, \H)$ of $\G$ and every 1-cocycle 
        $z \in Z^1 (\G, \pi)$ there exists a hermitian positive functional 
        $\phi : I(\G) \to \R$ such that $$\phi (1 - \g) = ||z(\g)||^2.$$
    \end{enumerate}
\end{prop}

We remark that these two constructions are inverses one of another. We refer to \cite[2.10.2]{bekka-harpe-valette} and \cite[Section 6]{ozawa-substitute} for details. 
The main result of this section is a generalization of the above fundamental relation between  positive functionals and
cocycles with unitary coefficients, to higher degrees. 
Our next proposition is an intermediate form of the higher GNS construction.

\begin{prop} \label{higher GNS}
Let $\G$ be a discrete group of type $F_{n+1}$ for some $n \in \Z_{\geq 1}$.
    The following conditions hold:
\begin{enumerate}
 \item[(1)] For every positive hermitian functional 
    $\psi:\mathbb{M}_{k_n}(\R \G)\to \R$
    which vanishes on $\Im D_n$ there exists 
    a unitary representation $\pi$ of $\G$ and a cyclic cocycle $z\in Z^n(\G,\pi)$ such that 
    $$\psi(a) = \langle \pi(a) z, z \rangle,$$
    for every $a\in \mathbb{M}_{k_n}(\R\G)$.
    
    \item[(2)] For every unitary representation $\pi$ of $\G$ 
    and every cocycle $z \in Z^n(\G, \pi)$, the positive hermitian functional $\psi: \mathbb{M}_{k_n}(\R\G) \to \R$ defined by
    $$\psi(a) = \langle \pi(a) z, z \rangle,$$ 
    for every $a \in \mathbb{M}_{k_n}(\R \G)$, vanishes on $\mathcal{SC}_n$.
   
\end{enumerate}
\end{prop}

\begin{proof}
$(1)$ Let $\psi: \mathbb{M}_{k_n}(\R \G)\to \R $ be a positive hermitian functional. By the GNS construction given in Proposition \ref{GNS for matrices of group algebras} there exists a unitary representation $(\pi, \H)$ of $\G$ and a cyclic cochain $z \in C^n(\G,\pi)$ such that for every $a \in \mathbb{M}_{k_n}(\R \G)$ we have:
   \begin{equation*}
       \psi(a) =  \langle \pi(a) z, z \rangle.
   \end{equation*}
Evaluating $\psi$ on  the element $\D_{n}^+=d_n^*d_n \in \Im D_{n}$, we obtain:
   \begin{equation*}
       0= \psi(\D_n^+) =  \langle \pi(\D_n^+) z, z\rangle = ||\pi(d_n) z||^2.
   \end{equation*}
    This shows that $z$ is indeed a cocycle. 
    
    $(2)$  Let $z \in \ker \pi(d_n)$ and $a\in \mathcal{SC}_n$. There exist rectangular matrices $x,y$ such that
    $a=d^*_n x+yd_n$. Therefore 
    \begin{equation*}
        \psi(a) = \langle \pi(a)z,z \rangle =\left\langle \pi(x)z,\pi(d_n)z\right\rangle + \left\langle \pi(y) \pi(d_n)z,z\right\rangle =0.
    \end{equation*}
\end{proof}

Consider the following diagram.
\begin{center}
\begin{tikzcd}
   \mathbb{M}_{k_{n-1}}(\R \G) 
   &  \operatorname{im} D_{n-1} \arrow[l,hook] \arrow{rd}{\phi}
   & \arrow{l}{\ \ \ D_{n-1}} \mathbb{M}_{k_n} (\R\G) \arrow{d}{\psi}\\
   && \R
\end{tikzcd}
\end{center}
Observe that if a functional $\psi:\mathbb{M}_{k_n}(\R\G) \to \R$ vanishes on $\ker D_{n-1}$, then it factors through $\operatorname{im}D_{n-1}$, i.e. it is of the form 
$$\psi=\phi\circ D_{n-1},$$
for some functional $\phi:\operatorname{im} D_{n-1}\to \R$. Indeed, 
setting, $\phi(x)=\psi(a)$ for any $a\in \mathbb{M}_{k_n}(\R\G)$ is a well-defined
functional. 
Moreover, $\phi$ is positive if and only if $\psi$ is positive,
by Proposition \ref{SOS in Im D_k are images of SOS}. Therefore 
Proposition \ref{higher GNS} can be expressed via the existence of positive functionals on $\operatorname{im} D_{n-1}$ as follows.

\begin{thm}[Higher cocycle GNS] \label{higher GNS - image version} 
Let $\G$ be a discrete group of type $F_{n+1}$ for some $n \in \Z_{\geq 1}$.
    The following conditions hold:
\begin{enumerate}
 \item[(1)] For every positive hermitian functional 
    $\phi:\operatorname{im} D_{n-1}\to \R$
    there exists 
    a unitary representation $\pi$ of $\G$ and a cyclic cocycle $z\in Z^n(\G,\pi)$ such that 
    $$\phi(d_{n-1}^* ad_{n-1}) = \langle \pi(a) z, z \rangle,$$
    for every $a\in \mathbb{M}_{k_n}(\R\G)$.
    
    \item[(2)] Suppose that $\mathcal{SC}_n = \ker D_{n-1}$. For every unitary representation $\pi$ of $\G$ 
    and every cocycle $z \in Z^n(\G, \pi)$ the functional $\phi: \operatorname{im} D_{n-1} \to \R$, defined by the formula
    $$\phi(d_{n-1}^*ad_{n-1}) = \langle \pi(a) z, z \rangle,$$ 
    for every $a \in \mathbb{M}_{k_n}(\R \G)$, is well-defined, positive and hermitian. 

\end{enumerate}
\end{thm}

   Observe that the above two constructions are inverses of one another. Namely, if we start with  a unitary representation $(\pi, \H)$ of $\G$, a cocycle $z \in Z^n(\G, \pi)$, we first consider the functional $\phi: \operatorname{im} D_{n-1} \to \R$ constructed in $(2)$ and then we consider the unitary representation $(\pi', \H')$ and the cyclic cocycle $z'$ obtained by applying $(1)$ to $\phi$. By Proposition \ref{GNS for matrices of group algebras}, there exists a unitary operator $U : \H' \to \H $ intertwining $\pi$ and $\pi'$ mapping entries of $z'$ to entries of $z$.

\begin{rem}\normalfont
Note that for $n=1$, Theorem \ref{higher GNS - image version} gives the classical Gelfand-Naimark-Segal construction of Theorem 
\ref{Classical conditionally negative GNS}. Indeed, when restricting to hermitian matrices
we have 
$$\Im (D_0|_{\mathbb{M}_{k_1} (\R \G)^h}) = I(\G)^h,$$ 
where $I (\G)^h$ denotes the self-adjoint elements in the augmentation ideal of $\R \G$, which is generated by the set $\{ 2 - \g - \g^{-1}: \g \in \G \}$ (for the non-hermitian case, compare with \cite[Lemma 4]{mizerka-nowak}).

Note that $\Im (D_0|_{\mathbb{M}_{k_1} (\R \G)^h}) \subseteq I(\G)^h$ is evident. To see the converse inclusion, fix a finite symmetric generating set $S$ of $\G$. For $\g \in \G$ fix a decomposition $$\g = s_1 \cdots s_n$$ of $\g$ as a word in $S$. For $s \in S$ let 
$$\g(s) = \{ i, s_i = s \}$$ 
and for $s, t \in S$ define 
\begin{equation*}
    a_t = \sum_{i \in \g(t)}  s_1 \cdots s_{i-1}. 
\end{equation*}
Then the horizontal vector 
$a = (a_t)_{t \in S}$ satisfies 
$1- \g = a d_0$
and consequently $$(1-\g)^* (1- \g) = d_0^* a^* a d_0.$$ 
If $\phi : I(\G) \to \R$ is a positive hermitian functional, let $(\pi, \H)$ be the unitary $\G$-representation and 
$z \in Z^1(\G, \pi)$ be the cocycle given by Proposition \ref{higher GNS - image version} $(1)$, satisfying $\phi (d_0^* a d_0) = \langle \pi(a) z, z \rangle$. Then
\begin{align*}
    2 \phi (1 - \g) & = \phi (2 - \g - \g^{-1}) = \phi (d_0^* a^*a d_0)\\
    &= \langle \pi(a^*a) z, z \rangle = \Vert\pi(a)z\Vert^2 \\
    & = \left\Vert \sum_{s \in S} \sum_{i \in \g(s)} \pi(s_1\ldots s_{i-1}) z(s) \right\Vert^2\\
    &= \left\Vert \sum_{i =1}^n \pi(s_1\ldots s_{i-1}) z(s_i) \right\Vert^2 = \left\Vert z(\g)\right\Vert^2.
\end{align*}

\end{rem}

We now turn to give a characterization of vanishing of higher unitary cohomology in terms of 
an extension property for positive functionals. 
The first observation is that functionals induced by coboundaries can be extended to the whole matrix algebra.


\begin{prop} \label{Coboundaries vs extending functionals}
For every positive hermitian functional 
    $\phi:\operatorname{im} D_{n-1}\to \R$, if the associated cocycle $z$ in Theorem \ref{higher GNS - image version} $(1)$ is a coboundary then the functional $\phi$ can be extended to a positive hermitian functional $\overline{\phi}:\mathbb{M}_{k_{n-1}}(\R\G)\to\R$.

\end{prop}

\begin{proof}
Given the functional $\phi:\operatorname{im} D_{n-1}\to \R$,
such that $\phi(d_{n-1}^*ad_{n-1})=\langle \pi(a)z,z\rangle$,
assume that $z=d_{n-1}c$ for some $c\in C^{n-1}(\G, \pi)$.
Let $\overline{\phi}:\mathbb{M}_{k_{n-1}}(\R\G)\to \R$ be defined by 
$$\overline{\phi}(m): =\langle \pi(m)c,c\rangle$$
for $m \in \mathbb{M}_{k_{n-1}}(\R\G)$.
For $m=d_{n-1}^*ad_{n-1}$ we have
\begin{align*}
    \overline{\phi}(d_{n-1}^*ad_{n-1}) &= \langle \pi(ad_{n-1})c,\pi(d_{n-1})c\rangle\\
    &= \langle \pi(a)z,z\rangle\\
    &=\phi(d_{n-1}^*ad_{n-1}).
\end{align*}
\end{proof}

\begin{rem} \normalfont
After this paper was posted on the arXiv, Uri Bader and Saar Bader have informed us that they have obtained independently a statement very similar to Proposition 
\ref{Coboundaries vs extending functionals} which includes a converse statement.
\end{rem}

The following Theorem shows that if one assumes the conclusion of Proposition \ref{Coboundaries vs extending functionals} for all positive hermitian functionals $\phi : \Im D_{n-1} \to \R$, then all their associated cocycles are coboundaries.
This gives a characterization of vanishing of unitary cohomology as a property of extension of functionals.

\begin{thm} \label{Vanishing of cohomology vs extension of functionals}
    Let $\G$ be a group of type $F_{n+1}$ for some $n \in \Z_{\geq 1}$. Suppose that $\mathcal{SC}_n = \ker D_{n-1}$. The following conditions are equivalent:
    \begin{enumerate}
        \item[(1)]  $H^n (\G, \pi) = 0$ for every unitary representation $(\pi, \H)$ of $\G$;
    \item[(2)] Every positive hermitian functional $\phi:\operatorname{im} D_{n-1} \to \R$ can be extended to a positive hermitian functional $\overline{\phi}:\mathbb{M}_{k_{n-1}}(\R\G)\to\R$.
    \end{enumerate}
\end{thm}

\begin{proof}
    The implication $(1) \implies (2)$ is a direct consequence of Proposition \ref{Coboundaries vs extending functionals} $(1)$. We will now show the converse.
    
    Suppose that condition $(2)$ holds, let $(\pi, \H)$ be a unitary representation of $\G$ and $z \in Z^n (\G, \pi)$ be a harmonic cocycle. Using Theorem \ref{higher GNS - image version} $(2)$, there is a positive hermitian functional $\phi:\operatorname{im} D_{n-1} \to \R$ such that 
    $$\phi(d_{n-1}^*ad_{n-1}) = \langle \pi(a) z, z \rangle$$ 
    for every $a \in \mathbb{M}_{k_n}(\R \G)$. By hypothesis, this functional can be extended to a positive hermitian functional $\overline{\phi}:\mathbb{M}_{k_{n-1}}(\R\G)\to\R$. Proposition \ref{GNS for matrices of group algebras} yields a unitary representation $(\pi', \H')$ of $\G$ and a cochain $c \in (\H')^{k_{n-1}}$ such that
    $$ \overline{\phi}(a) = \langle \pi'(a) c, c \rangle $$
    for every $a \in \mathbb{M}_{k_{n-1}}(\R \G)$. We have
    $$ 0 = \langle \pi(\D_n) z, z \rangle = \phi ((\D_{n-1}^+)^2 ) = \langle \pi'(\D_{n-1}^+)^2 c, c \rangle = ||\pi'(\D_{n-1}^+)c||^2. $$
    Hence $\pi'(\D_{n-1}^+)c = 0$ and 
    $$ ||z||^2 = \phi(\D_{n-1}^+) = \langle \pi'(\D_{n-1}^+) c, c \rangle = 0.  $$
    Thus we have shown that $\overline{H}^n(\G, \pi ) = \ker \pi (\D_n) =  0$ for every unitary representation $\pi$ of $\G$. Using the argument in \cite[Theorem 3.7]{bader-sauer}, we obtain that $H^n(\G, \pi ) = 0$ for every unitary representation $\pi$ of $\G$.
\end{proof}

\section{Reducedness of cohomology and sums of squares}

In this section we will consider algebraic conditions,
involving sums of squares,
for the cohomology with unitary coefficients 
being reduced or for it to vanish.
The following result, proved in \cite{bader-nowak} characterizes when cohomology is reduced for every unitary representation (here we state this differently, the equivalence follows from Positivstellensatz \cite[Theorem 10.20]{schmudgen-book}, and a standard direct sum argument for the uniformity of $\l$). 

\begin{thm}\cite[Proposition 16 3.]{bader-nowak} \label{Characterization of reducedness by Bader-Nowak}
    Let $\G$ be a discrete group of type $F_{n+1}$. The following conditions are equivalent: 
    \begin{enumerate}
        \item[(1)] For every unitary representation $(\pi, \H)$ of $\G$, the cohomology spaces $H^n (\G, \pi)$ are Hausdorff.
    \item[(2)] There exists $\l > 0$ such that
    \begin{equation*}
        (\D_{n-1}^+)^2 - \l \D_{n-1}^+ \in \overline{\S^2 \mathbb{M}_{k_{n-1}}(\R\G)}.
    \end{equation*}
    \end{enumerate}
\end{thm}

In view of the above characterization of 
cohomology with unitary coefficients being reduced,
it is a natural question if the property 
\begin{equation*}
        (\D_{n-1}^+)^2 - \l \D_{n-1}^+ \in \S^2\mathbb{M}_{k_{n-1}}(\R\G)
    \end{equation*}
is equivalent to the one given in the previous result. 
As it surprisingly turns out, when the conclusion of Theorem \ref{theorem: ker D_i=SC_i} holds, this property characterizes cohomological vanishing, which is strictly stronger than reducedness of cohomology for $n \geq 2$.

\begin{thm}\label{theorem: Higher Ozawa equation is equivalent to Higher (T)}
    Let $\G$ be a discrete group of type $F_{n+1}$ for some $n \in \Z_{\geq 1}$ and suppose that $\ker D_{n-1} = \mathcal{SC}_n$. The following conditions are equivalent: 
    \begin{enumerate}
        \item[(1)] For every unitary representation $(\pi, \H)$ of $\G$, we have $H^{n} (\G, \pi) = 0$;
    \item[(2)] There exists $\l > 0$ such that
    \begin{equation*}
        \D_{n} - \l I \in \S^2 \mathbb{M}_{k_n}(\R\G);
    \end{equation*}
    \item[(3)] There exists $\l > 0$ such that
    \begin{equation*}
        (\D_{n-1}^+)^2 - \l \D_{n-1}^+ \in \S^2  \mathbb{M}_{k_{n-1}}(\R\G).
    \end{equation*}
        \end{enumerate}
\end{thm}

\begin{proof}
The equivalence between the conditions $(1)$ and $(2)$ is \cite[Main Theorem]{bader-nowak} combined with \cite[3.11]{bader-sauer}. We will now show that these two conditions are equivalent to $(3)$.

The implication $(2) \implies (3)$ follows from Lemma \ref{Squares of non-square matrices} and the equality
    \begin{equation*}
        (\D_{n-1}^+)^2 - \l \D_{n-1}^+ = d_{n-1}^* (\D_n -\l I)d_{n-1} = D_{n-1}(\D_n - \l I).
    \end{equation*}

We will now show $(3) \implies (1)$. 
Let $(\pi, \H)$ be a unitary representation of $\G$ and $z \in \ker \pi (\D_{n})$ be a harmonic cochain. In particular, $z$ is a cocycle and Theorem \ref{higher GNS - image version} $(2)$ (which we can use since we are assuming $\ker D_{n-1} = \mathcal{SC}_n$) implies that there exists a positive functional $\phi : \operatorname{im} D_{n-1} \to \R$ such that
\begin{equation*}
    \phi(d_{n-1}^* a d_{n-1}) = \langle \pi(a) z, z \rangle,
\end{equation*}
for every $a \in  \mathbb{M}_{k_n}(\R \G)$. Since by assumption $(\D_{n-1}^+)^2 - \l \D_{n-1}^+ \in \S^2  \mathbb{M}_{k_{n-1}}(\R\G)$ and $\phi$ is positive, we have
\begin{equation*}
    \phi\left((\D_{n-1}^+)^2 - \l \D_{n-1}^+\right) \geq 0.
\end{equation*}
On the other hand, harmonicity of $z$ gives
\begin{equation*}
    \phi\left((\D_{n-1}^+)^2 - \l \D_{n-1}^+\right) = \langle \pi (\D_n -\l I) z, z \rangle = - \l ||z||^2.
\end{equation*}
These two conditions force $z= 0$. 
It follows that 
$$\overline{H}^n(\G, \pi) = \ker \pi (\D_n) = 0.$$ 
However, by Theorem \ref{Characterization of reducedness by Bader-Nowak} the assumed condition $(3)$ implies that cohomology of $\G$ with any unitary coefficients is reduced. Hence
\begin{equation*}
    H^n(\G, \pi) = \overline{H}^n(\G, \pi) = 0.
\end{equation*}

\end{proof}

\begin{rem}\normalfont
    For $n =1$ the main idea behind the proof of Theorem \ref{theorem: Higher Ozawa equation is equivalent to Higher (T)} was outlined in \cite[p. 2554]{ozawa-substitute}.
\end{rem}

\begin{cor}
    Let $\G$ be a cocompact lattice in $\mathrm{SL}_{r+1}(\mathbb{Q}_p)$ with $r \geq 2$ and $p$ a prime (or more generally in any simple algebraic group of rank $r$ defined over a non-Archimedean local field). There exists a constant $\l > 0$ such that
    \begin{equation*}
        \D_{r-1} - \l I \in \S^2 \mathbb{M}_{k_{r-1}}(\R\G), \qquad (\D_{r-1}^+)^2 - \l \D_{r-1}^+ \in \overline{\S^2 \mathbb{M}_{k_{r-1}}(\R\G)}.
    \end{equation*}
    On the other hand, for any $\l > 0$ we have:
    \begin{equation*}
        (\D_{r-1}^+)^2 - \l \D_{r-1}^+ \notin \S^2 \mathbb{M}_{k_{r-1}}(\R\G).
    \end{equation*}
\end{cor}

\begin{proof}
The results of Garland \cite{garland} and Casselman \cite{casselman} (see \cite{grinbaum-reizis-oppenheim} for a shorter proof) show that for any cocompact lattice $\G$ in a simple algebraic group $G$ of rank $r$ defined over a non-Archimedean local field we have
\begin{equation*}
    H^n(\G, \pi) = 0
\end{equation*}
for every unitary representation $\pi$ of $\G$ and $1 \leq n \leq r-1$. 
At the same time there exist unitary representations 
$\rho$ such that 
\begin{equation*}
    H^r(\G, \rho) \neq 0.
\end{equation*}
Bader and Sauer showed that the condition $ H^{n-1}(\G, \pi) = 0$ for every unitary representation $\pi$ of $\G$, implies that $H^n(\G, \pi)$ is Hausdorff for every unitary representation $\pi$ of $\G$ \cite[3.11]{bader-sauer}. 

The first two equations in our statement follow from Theorem \ref{Characterization of reducedness by Bader-Nowak} and \cite[Main Theorem]{bader-nowak} as they are the group algebraic translations of the corresponding vanishing and reducedness results. 

The lattice $\G$ has geometric dimension $r \geq 2$. Theorem \ref{theorem: ker D_i=SC_i} says that $\ker D_{n-1} = \mathcal{SC}_n$ for $n > r/2$ (so in particular for $n = r$) and hence Theorem \ref{theorem: Higher Ozawa equation is equivalent to Higher (T)} says that $ H^r(\G, \rho) \neq 0$ implies that for any $\l >0$ we have 
\begin{equation*}
   (\D_{r-1}^+)^2 - \l \D_{r-1}^+ \notin \S^2 \mathbb{M}_{k_{r-1}}(\R\G).
\end{equation*}
\end{proof}

\section{Higher Property $H_T$}

We now define property $H_T^n$ and give two characterizations of it in terms of equations in the group algebra (though we need closures to formulate them). The first one holds in general for all groups of type $F_{n+1}$, while the second one requires the conclusion of Theorem \ref{higher GNS - image version}.

The following definition appeared implicitly in the work of Shalom \cite{shalom-harmonic} and explicitly in \cite[2.1.1]{gotfredsen-kyed}. The following discussion is relevant in the broader context of locally compact groups, for consistency here we restrict to discrete groups.

\begin{defn}
    Let $\G$ be a discrete group and $n \in \Z_{>0}$. We say that $\G$ has \textit{property $H_T^n$} if for every unitary representation $(\pi, \H)$ of $\G$ without $\G$-invariant vectors, we have $\overline{H}^n(\G, \pi) = 0$.
\end{defn}

Property $H_T^1$ corresponds to the more classical property $H_T$ defined by Shalom \cite[p. 125]{shalom-harmonic}. For finitely generated groups, property $(T)$ is equivalent to having property $H_T$ and finite abelianization. Other examples of groups satisfying property $H_T$ are often amenable, such as finitely generated nilpotent groups and wreath products $F \wr \Z$, where $F$ is finite \cite[1.14]{shalom-harmonic}.

In a similar fashion, higher versions of property $(T)$ imply the corresponding higher versions of property $H_T$. Namely, properties $(T_n)$ and $[T_n]$ \cite{bader-sauer}, imply property $H_T^k$ for all $1 \leq k \leq n$. Finitely generated nilpotent groups satisfy property $H_T^n$ for every $n \geq 1$ \cite[10.5]{blanc} \cite[p. 243]{guichardet-book}. A discrete group of type $FP_{n+1}$ has Property $[T_n]$ defined by Bader-Sauer if and only if it has Property $H_T^k$ and satisfies $H^k(\G, \R) = 0$ for $1 \leq k \leq n$.

We come back to the usual setting: $\G$ denotes a group of type $F_n$ for some $n \geq 1$ and we use notation from previous sections. For any $k \geq 0$ we define
\begin{equation*}
    \eta_k = \mathrm{diag}(\D_0, \ldots, \D_0) \in \mathbb{M}_{k}(\R \G).
\end{equation*}

The following lemma explains the relation between the matrix $\eta_k$ and the existence of invariant vectors.

\begin{lem}\label{operator eta}
    Let $(\pi, \H)$ be a unitary representation of $\G$ and $c \in \H^{k}$. We have:
    $\langle \pi(\eta_k)c, c \rangle = 0$ if and only if $c \in (\H^{\pi(\G)})^{k}$.
\end{lem}

\begin{proof}
    By writing $c = (c_1, \ldots, c_{k})^T \in \H^k$ we see that
    \begin{equation*}
        \langle \pi(\eta_k)c, c \rangle = \sum_{i=1}^k \langle \pi(\D_0)c_i, c_i \rangle = \sum_{i=1}^k || \pi(d_0)c_i||^2 \geq 0,
    \end{equation*}
    and these are equal to $0$ if and only if $c_i \in \ker \pi(d_0)$ for every $i=1, \ldots, k$. If $S$ denotes the finite generating set of $\G$ we used to define $d_0$, we have $\ker \pi(d_0) = \H^{\pi(S)} = \H^{\pi(\G)}$, hence $\langle \pi(\eta_k)c, c \rangle = 0$ if and only if $c_i \in  \H^{\pi(\G)}$ for every $i=1, \ldots, k$.
\end{proof}

We obtain the following characterization of Property $H_T^n$.

\begin{thm}\label{Characterization of higher H_T}
  Let $n \geq 1$. A discrete group $\G$ of type $F_{n+1}$ has property $H_T^n$ if and only if \begin{equation*}
      - \eta_{k_n} \in \overline{\S^2 \mathbb{M}_{k_n}(\R \G) + \R \D_n},
  \end{equation*}
  or equivalently, if and only if for every $\varepsilon >0$ there exists $R>0$ such that
  \begin{equation*}
      R \D_n - \eta_{k_n} + \varepsilon I \in \S^2 \mathbb{M}_{k_n}(\R \G).
  \end{equation*}
\end{thm}

\begin{proof} 

    Suppose first that $- \eta_{k_n} \in \overline{\S^2 \mathbb{M}_{k_n}(\R \G) + \R \D_n}$. Then for every $\varepsilon>0$, there exists $R>0$ such that
\begin{equation*}
    -\eta_{k_n} + \varepsilon I_{k_n} + R \D_n \in \S^2 \mathbb{M}_{k_n}(\R \G).
\end{equation*}
Let $(\pi, \H)$ be a unitary representation of $\G$ without $\G$-invariant vectors and let $c \in \ker \pi (\D_n)$. For every $\varepsilon >0$ we have
    \begin{equation*}
        \langle \pi(-\eta_{k_n} + \varepsilon I_{k_n})c, c \rangle \geq 0.
    \end{equation*}
    Hence $\langle \pi(\eta_{k_n})c, c \rangle \leq 0$. Since $\eta_{k_n} \in \S^2 \mathbb{M}_{k_n}(\R \G)$, Lemma \ref{operator eta} implies that $c \in (\H^{\pi(\G)})^{k}$. The hypothesis $\H^{\pi(\G)} = 0$ implies that $c = 0$. Equation \ref{reduced cohom = harmonic cocycles} shows that $\overline{H}^n (\G, \pi) = \ker \pi (\D_n) = \{ 0 \}$.

    Conversely suppose that $- \eta_{k_n} \notin \overline{\S^2 \mathbb{M}_{k_n}(\R \G) + \R \D_n}$. Then applying the Hahn-Banach theorem to the convex subset $ \overline{\S^2 \mathbb{M}_{k_n}(\R \G) + \R \D_n}$ in $\mathbb{M}_{k_n}(\R \G)^h$ implies that there exists a hermitian functional $\psi: \mathbb{M}_{k_n}(\R \G) \to \R$ such that
    \begin{equation*}
        \psi(- \eta_{k_n}) <0 \text{ and } \psi(\overline{\S^2 \mathbb{M}_{k_n}(\R \G) + \R \D_n}) \geq 0.
    \end{equation*}

    Since $\psi(\S^2 \mathbb{M}_{k_n}(\R \G)) \geq 0$, Proposition \ref{GNS for matrices of group algebras} gives a unitary representation $(\pi,\H)$ of $\G$ and a vector $c \in \H^{k_n}$ such that $\psi (a) = \langle \pi(a)c, c\rangle$ for every $a \in \mathbb{M}_{k_n}(\R \G)$. 
    
    Since $\psi (\eta_{k_n}) > 0$, Lemma \ref{operator eta} implies that $c \notin (\H^{\pi(\G)})^{k_n}$. This means that if $P : \H \to \H_0$ denotes the $\G$-equivariant orthogonal projection onto the quotient $\H_0:= \H / \H^{\pi(\G)}$, then in the quotient representation $(\pi_0, \H_0)$ of $(\pi, \H)$ (which has no $\G$-invariant vectors) we have $c_0 : = P(c) \neq 0$ in $\H_0^{k_n}$.

    Since $\psi(\R \D_n) \geq 0$, we necessarily have $\psi(\D_n) = 0$ as $\psi$ is linear. However, this means that
    \begin{equation*}
        0 = \psi(\D_n) = \langle \pi(\D_n)c, c\rangle = || \pi(d_{n-1}^*)c||^2 + ||\pi(d_n) c ||^2,
    \end{equation*}
    and $ \pi(d_{n-1}^*)c = \pi(d_n) c = 0$ and so $\pi (\D_n)c = 0$. Hence $\pi_0 (\D_n)c_0 = P(\pi(\D_n)c ) = 0 $. 
    
    This shows that $(\pi_0, \H_0)$ is a unitary representation of $\G$ without $\G$-invariant vectors for which $\overline{H}^n(\G, \pi_0) = \ker \pi(\D_n) \neq \{ 0 \}$ by Equation \ref{reduced cohom = harmonic cocycles}.
\end{proof}

 We now turn to another characterization of property $H_T^n$, which is precisely the higher-dimensional analogue of Ozawa's characterization of property $H_T$ \cite[Corollary 12]{ozawa-substitute}. We set

\begin{equation*}
    \square_n := d_{n}^* \eta_{k_{n+1}} d_{n} =  d_{n}^* \mathrm{diag}(\D_0, \ldots, \D_0) d_{n} \in \mathbb{M}_{k_n}(\R \G).
\end{equation*}

\begin{thm}\label{Higher analogue of Ozawa's characterization}
Let $\G$ be a discrete group of type $F_{n+1}$ for some $n \geq 1$, satisfying $\mathcal{SC}_n = \ker D_{n-1}$. Then $\G$ has property $H_T^{n}$ if and only if for every $\varepsilon >0$ there exists $R>0$ such that
  \begin{equation*}
      R (\D_{n-1}^+)^2 - \square_{n-1} + \varepsilon \D_{n-1}^+ \in \S^2 \mathbb{M}_{k_{n-1}}(\R \G).
  \end{equation*}
\end{thm}

\begin{proof}
    If the group $\G$ has property $H_T^n$, then, by  Theorem \ref{Characterization of higher H_T}, 
    for every $\varepsilon >0$, there exists $R >0$ such that we have
    \begin{equation*}
          R \D_n - \eta_{k_{n}} + \varepsilon I \in \S^2 \mathbb{M}_{k_{n}}(\R \G).
    \end{equation*}
    applying the map $D_{n-1}$ to this equation gives
    \begin{equation*}
        R (\D_{n-1}^+)^2 - \square_{n-1} + \varepsilon \D_{n-1}^+ \in D_{n-1}(\S^2 \mathbb{M}_{k_n}(\R \G)) \subseteq \S^2 \mathbb{M}_{k_{n-1}}(\R \G).
    \end{equation*}
    Conversely, suppose that for every $\varepsilon >0$, there exists $R >0$ such that 
  \begin{equation*}
      R (\D_{n-1}^+)^2 - \square_{n-1} + \varepsilon \D_{n-1}^+ \in \S^2 \mathbb{M}_{k_{n-1}}(\R \G).
  \end{equation*}
    Let $\pi$ be a unitary representation without $\G$-invariant vectors and $b \in \ker \pi (\D_{n})$ be a harmonic cocycle. Theorem \ref{higher GNS - image version} $(2)$ gives a hermitian functional $\phi:\operatorname{im} D_{n-1} \to \R$ such that for every $a \in  \mathbb{M}_{k_n}(\R \G) $ we have
\begin{equation*}
    \phi(d_{n-1}^* a d_{n-1}) = \langle \pi(a) b, b \rangle.
\end{equation*}
Fix $\varepsilon>0$ and choose $R>0$ so that $R (\D_{n-1}^+)^2 - \square_{n-1} + \varepsilon \D_{n-1}^+ \in  \S^2 \mathbb{M}_{k_{n-1}}(\R \G)$. Then
\begin{equation*}
    \phi(- \square_{n-1} + \varepsilon \D_{n-1}^+) = \langle \pi( R \D_{n} - \eta_{k_{n}} + \varepsilon I) b, b \rangle \geq 0.
\end{equation*}
Since $\square_{n-1} \in \S^2 \mathbb{M}_{k_{n-1}}(\R \G) $, we have $0 \leq \phi (\square_{n-1}) \leq \varepsilon \phi (\D_{n-1}^+)$. Since this holds for every $\varepsilon >0$, we have $  \langle \pi(\eta_{k_n}) b, b \rangle = \phi (\square_{n-1}) = 0$. By Lemma \ref{operator eta}, this means that $b \in (\H^{\pi(\G)})^{k_n}$. Since $\pi$ has no invariant vectors, we have $b= 0$. By Equation \ref{reduced cohom = harmonic cocycles} we have $\overline{H}^n(\G, \pi) = \ker \pi (\D_n) = 0$, hence $\G$ has property $H_T^n$.

\end{proof}

\section{Remarks and questions}

\subsection{Matricial homology and ideals}
In the above considerations it is natural to ask when do we have
$\ker D_{n-1}=\operatorname{im} D_n$. 
This property is equivalent to the fact that $\Delta_{n}^+=d_n^*d_n$
is an order unit in $\ker D_{n-1}$.
It was observed by P. Mizerka and the second author that in 
such a case $\ker D_{n-1}$ is an ideal in $\mathbb{M}_{k_n}(\R\G)$, that
plays a similar role as the augmentation ideal. 
However, there are simple examples where we observe that this is not the case.

\begin{example}
Let $\G =\Z^2 = \langle s, t \, | \, st= ts \rangle$ acting on its simply-connected 2-dimensional Cayley complex $X$. The number of $\G$-orbits on each dimension are $k_0 = 1, k_1 = 2$ and $k_2 = 1$. The codifferentials can be taken to be
\begin{equation*}
    d_0 = \begin{pmatrix}
        1-s \\
        1- t
    \end{pmatrix}, \quad d_1 = (1- t, s-1 )
\end{equation*}
(where the computation for $d_1$ was obtained using its characterization as Fox derivatives of relations by generators, as in \cite{kaluba-mizerka-nowak}).
If $a = \begin{pmatrix}
    a_{ss} & a_{st} \\
    a_{ts} & a_{tt}
\end{pmatrix} \in \mathbb{M}_2(\R \G)$, we let $D_0 (a) = d_0^* a d_0$. 
Then  $$\ker D_0 \neq \Im D_1.$$
Indeed, the condition $a \in \ker D_0$ means
\begin{align*}
    0 = d_0^* a d_0 & = (1-s^{-1}) a_{ss} (1-s) + (1-s^{-1})a_{st} (1-t) \\
    &+ (1-t^{-1})a_{ts} (1-s) + (1-t^{-1})a_{tt}(1-t).
\end{align*}
On the other hand for $\xi \in \R\G$ we have
\begin{align*}
    D_1(\xi) & = d_1^* \xi d_1 = \begin{pmatrix}
        1- t^{-1} \\
        s^{-1} -1
    \end{pmatrix} \xi \begin{pmatrix}
        1 -t & s^{-1}- 1
    \end{pmatrix} \\
    &  = \begin{pmatrix}
        (1-t^{-1}) \xi (1-t) &  (1-t^{-1}) \xi (s-1) \\
        (s^{-1} - 1) \xi (1-t) & (s^{-1} - 1) \xi (s-1)
    \end{pmatrix}.
\end{align*}
Hence if we choose a matrix $a\in \mathbb{M}_2(\R \G)$ of the form \begin{equation*}
    a = \begin{pmatrix}
        (1-t^{-1}) \xi (1-t) &  0 \\
        0  & (s^{-1} - 1) \xi (1-s)
    \end{pmatrix},
\end{equation*}
we obtain a matrix $a$ lying in $\ker D_0$ but outside of $\Im D_1$.
\end{example}

In view of this example, the natural question to ask is the following.

\begin{ques}
    Does there exist an infinite group $\G$ of type $F_{n+1}$ for some $n \in \na$ such that $\ker D_{n-1} = \Im D_n$?
\end{ques}

\subsection{Algebraic characterizations of reducedness}

Theorem \ref{theorem: Higher Ozawa equation is equivalent to Higher (T)} implies that elements of the form $(\D_{n-1}^+)^2 - \l \D_{n-1}^+$ do not provide an algebraic characterization of reducedness of cohomology. 
Nevertheless, reducedness could still be characterized by a similar condition involving other elements in the group algebra.

\begin{ques}
   Does there exist $a \in \mathbb{M}_k(\R \G)$ (or a sequence $(a_n)_{n \in \na}$ with $a_n \in \mathbb{M}_k(\R \G)$) such that the property that $H^n(\G, \pi)$ is Hausdorff for every unitary representation is equivalent to $a \in \S^2 \mathbb{M}_k(\R \G)$? (or that there exists $n \in \na$ such that $a_n \in \S^2 \mathbb{M}_k(\R \G)$?)
\end{ques}

\subsection{Split conjugation spaces and matricial homology}

We expect the conclusion of Theorem \ref{theorem: ker D_i=SC_i} to hold in greater generality.

\begin{ques}
    Let $\G$ be a group of type $F_{n+1}$. Do we have $\ker D_{n-1} = \mathcal{SC}_n$ in all dimensions?
\end{ques}

The condition $\ker D_{n-1} = \mathcal{SC}_n$ is our main tool to prove Theorem \ref{higher GNS - image version}. Even if its conclusion always holds for $n  = 1$, we do not know whether $\ker D_{0} = \mathcal{SC}_1$ holds for finitely presented groups.

\bibliographystyle{amsalpha}
\bibliography{refs.bib}

\noindent Antonio López Neumann \\
Université Paris Cité, Sorbonne Université, CNRS, IMJ-PRG, F-75013 Paris, France \\
lopezneumann@imj-prg.fr \\

\noindent Piotr W. Nowak \\
Institute of Mathematics of the Polish Academy of Sciences (IMPAN), Warsaw \\ 00-656 Warsaw, Poland \\
pnowak@impan.pl \\

\end{document}